\documentclass[twocolumn]{autart}    

\usepackage[colorlinks,linkcolor=red,anchorcolor=blue,citecolor=green]{hyperref}

\usepackage{tikz}
\usepackage{amsmath}
\usepackage{latexsym, amssymb, amsmath, amsbsy,amsopn, amstext}
\usepackage{graphicx}
\usepackage{hyperref}
\usepackage{cancel, color}
\usepackage{epstopdf}

\DeclareMathOperator{\lcm}{lcm}
\def\lplus{\rotatebox[]{-90}{$\pm$}}
\def\rplus{\rotatebox[]{90}{$\pm$}}
\def\lminus{\vdash}
\def\rminus{\dashv}
\def\cal{\mathcal}

\def\a{\alpha}

\def\d{\delta}
\def\e{\epsilon}

\def\0{{\bf 0}}
\def\A{\left<A\right>}
\def\B{\left<B\right>}

\newcommand{\Q}{{\mathbb Q}}

\newcommand{\Z}{{\mathbb Z}}
\newtheorem{dfn}[thm]{Definition}
\newtheorem{prp}[thm]{Proposition}

\begin{document}

\begin{frontmatter}

\title{Basis for the linear space of matrices under equivalence\thanksref{footnoteinfo}}

\thanks[footnoteinfo]{ This work was supported partly by 
Fundamental Research Funds for the Central Universities (No. HEUCF160404) and
Natural Science Foundation of Heilongjiang Province of China (No. LC2016023).}

\author{Kuize Zhang}

\address{College of Automation, Harbin Engineering University, 150001, Harbin, PR China}

\begin{keyword}
identity equivalence, semi-tensor addition, semi-tensor product, algebraic structure, basis
\end{keyword}

\begin{abstract}
	The semi-tensor product (STP) of matrices which was proposed by Daizhan Cheng in 2001 \cite{che01},
	is a natural generalization
	of the standard matrix product and well defined at every two finite-dimensional matrices.
	In 2016, Cheng proposed a new concept of semi-tensor addition (STA) which is a natural generalization of 
	the standard matrix addition and well defined at every two finite-dimensional matrices with the same ratio
	between the numbers of rows and columns \cite{che16STA}. In addition, an identify equivalence relation
	between matrices
	was defined in \cite{che16STA}, STP and STA were proved valid for the corresponding identify equivalence classes,
	and  the corresponding quotient space was endowed with an algebraic structure and a manifold structure.
	In this follow-up paper, we give a new concise basis for the quotient space, which also shows 
	that the Lie algebra corresponding to the quotient space is of countably infinite dimension.
\end{abstract}

\end{frontmatter}

\section{Introduction}

In this section, we introduce preliminaries which have been shown in \cite{che01,che16STA}.

\subsection{Semi-tensor product}

In this paper, we use ${\cal M}_{m\times n}$ to denote the set of $m\times n$ real matrices,
where $m,n$ are two positive integers. 

\begin{dfn}\label{d3.1} Let ~$A\in {\cal M}_{m\times n}$,  $B\in {\cal M}_{p\times q}$. Denote by $t=\lcm\{n,p\}$, the least common multiple of $n$ and $p$. Then
\begin{enumerate}
\item the left semi-tensor product (STP) of
$A$ and $B$ is defined as
\begin{align} \label{3.1} A\ltimes B:=\left(A\otimes I_{t/n}\right)\left(B\otimes I_{t/p}\right),
\end{align}
\item the right STP of
$A$ and $B$ is defined as
\begin{align} \label{3.2} A\rtimes B:=\left(I_{t/n}\otimes A\right)\left( I_{t/p} \otimes B\right),
\end{align}
\end{enumerate}
\end{dfn}
where $\otimes$ is the Kronecker product \cite{hor91}, \cite{zha04}.

In the following we only discuss the left STP, and briefly call it STP.
\begin{rem}
	For every $A\in{\cal M}_{m\times n}$ and $B\in{\cal M}_{p\times q}$, 
	$A\ltimes B\in{\cal M}_{(mt)/n\times (qt)/p}$, where $t=\lcm(n,p)$.
\end{rem}

\subsection{Semi-tensor addition}

Denote
$$
{\cal M}_{\mu}:=\left\{M\in {\cal M}_{m\times n}\;|m,n\in\Z_{+},\;m/n=\mu \right\},\quad \mu\in \Q_+,
$$
where $\Z_{+}$ and $\Q_{+}$ denote the sets of positive integers and positive rational numbers, respectively.
Denote
\begin{align}\label{2.1}
{\cal M}=\bigcup_{\mu\in \Q_+}{\cal M}_{\mu}.
\end{align}

Note that the right hand side of (\ref{2.1}) is a partition of the left hand side of 
\eqref{2.1}. That is,
$$
{\cal M}_{\mu_1}\cap {\cal M}_{\mu_2}=\emptyset, \quad \forall \mu_1\neq \mu_2.
$$

\begin{dfn}\label{d2.5} Let $A,~B\in {\cal M}_{\mu}$, where $\mu\in\Q_{+}$.
	Precisely, $A\in {\cal M}_{m\times n}$ and $B\in {\cal M}_{p\times q}$, and $m/n=p/q=\mu$. Set $t=\lcm\{m,p\}$.
\begin{enumerate}
\item The left semi-tensor addition (STA) of $A$ and $B$, denote by
$\lplus$, is defined as
\begin{align}\label{2.2}
A\lplus B:=\left(A\otimes I_{t/m}\right)+\left(B\otimes I_{t/p}\right).
\end{align}
We also denote the left semi-tensor subtraction as
\begin{align}\label{2.3}
A\lminus B:=A\lplus (-B).
\end{align}
\item The right STA of $A$ and $B$, denote by
$\rplus$, is defined as
\begin{align}\label{2.4}
A\rplus B:=\left(I_{t/m}\otimes A\right)+\left(I_{t/p}\otimes B\right).
\end{align}
We also denote the right semi-tensor subtraction as
\begin{align}\label{2.5}
A\rminus B:=A\rplus (-B).
\end{align}
\end{enumerate}
\end{dfn}

\begin{rem}\label{r2.6} Let $\sigma\in \{\lplus, \lminus, \rplus, \rminus\}$ be one of the four binary operations,
	and $\mu_1$ and $\mu_2$ two positive rational numbers.
\begin{enumerate}
	\item If $\mu_1=\mu_2=:\mu$, then for all $A,~B\in {\cal M}_{\mu}$, $A\sigma B\in {\cal M}_{\mu}$.
	\item For all $A\in{\cal M}_{\mu_1}$ and $B\in{\cal M}_{\mu_2}$, $A\ltimes B\in{\cal M}_{\mu_1\mu_2}$.
	\item If $A$ and $B$ are as in Definition \ref{d2.5}, then $A\sigma B\in {\cal M}_{t\times \frac{nt}{m}}$.
\end{enumerate}
\end{rem}

Similar to the case for STP, 
in the following we only discuss the left STA, and briefly call it STA.

\subsection{Identity equivalence relation}

\begin{dfn}\label{d2.1} Let $A,~B\in {\cal M}$ be two matrices.
\begin{enumerate}
	\item $A$ and $B$ are said to be left identity equivalent (LIE), denoted by $A\sim_{\ell} B$, if there exist two identity matrices $I_s$, $I_t$, $s,t\in \Z_{+}$, such that
$$
A\otimes I_s=B\otimes I_t.
$$
\item  $A$ and $B$ are said to be right identity equivalent (RIE), denoted by $A\sim_r B$, if there exist two identity matrices $I_s$, $I_t$, $s,t\in \Z_{+}$, such that
$$
I_s\otimes A=I_t\otimes B.
$$
\end{enumerate}
\end{dfn}

\begin{rem}\label{r2.2} It is easy to verify that the LIE $\sim_{\ell}$ (similarly, RIE $\sim_r$)
	is an equivalence relation. That is, it is (i) reflexive ($A\sim_{\ell} A$);
	(ii) symmetric (if $A\sim_{\ell} B$, then $B\sim_{\ell} A$);
	and (iii) transitive (if $A\sim_{\ell} B$, and  $B \sim_{\ell} C$, then  $A \sim_{\ell} C$).
\end{rem}

In the sequel we only consider the left LIE, call it LIE for short, and denote it briefly by $\sim$.

\begin{dfn}\label{d2.3}  We are given a matrix $A\in {\cal M}$.
\begin{enumerate}
\item The equivalent class of $A$, denoted by
$$
\A:=\left\{B\;|\;B\sim A\right\}.
$$
\item $A$ is called reducible, if there is an identify matrix $I_s$, where $s\geq 2$, and a matrix $B$, such that
$A=B\otimes I_s$. Otherwise, $A$ is called irreducible.
\end{enumerate}
\end{dfn}

\begin{thm}\label{t2.4}
	For every $\bar A\in{\cal M}$, in $\left<\bar A\right>$ there exists a unique
	$A_0\in\left<\bar A\right>$ such that $A_0$ is irreducible.
\end{thm}

By Theorem \ref{t2.4}, for each $\A$ 
and all matrices $B\in{\cal M}_{m\times n}$ and $C\in
{\cal M}_{p\times q}$ both in $\A$,
one has $m/n=p/q$. Then the matrix in $\A$ that has the minimal number 
of rows is the unique irreducible element of $\A$.

The following corollary reveals the structure of identity equivalence classes.

\begin{cor}\label{c2.400} Let $A_0$ be the only irreducible element in $\A$. Then
\begin{align}\label{2.05}
\A=\{A_0\otimes I_s\;|\;s=0,1,2,\cdots\}.
\end{align}
\end{cor}

\begin{thm}\label{t2.8}
	Consider the algebraic system $({\cal M}_{\mu},\lplus)$, where $\mu\in\Q_{+}$.
	The LIE $\sim$ is a congruence with respect to $\lplus$.
\end{thm}

\begin{proof}
	Arbitrarily chosen $\tilde{A},A,\tilde{B},B\in{\cal M}_{\mu}$, we assume $\tilde{A}\sim A$ and $\tilde{B}\sim B$. 
	To prove this theorem, we need to prove $\tilde{A}\lplus \tilde{B}\sim A\lplus B$.
	We give an alternative proof compared to the one in \cite{che16STA}.

	By Theorem \ref{t2.4} and Corollary \ref{c2.400}, there exist unique irreducible matrices
	$A_0\in{\cal M}_{\mu n\times n}$ and $B_0\in{\cal M}_{\mu m\times m}$ such that 
	\begin{align}
		\begin{array}{cc}
			\tilde{A} = A_0 \otimes I_s,&A = A_0 \otimes I_t,\\
			\tilde{B} = B_0 \otimes I_p,&B = B_0 \otimes I_q\\
		\end{array}
	\end{align}
	for some positive integers $s,t,p,q$.
	Set $T:=\lcm(ns,mp)$, $S:=\lcm(nt,mq)$, and $R:=\lcm(n,m)$. Then we have
	\begin{equation}\label{2.605}
		\begin{split}
			\tilde {A} \lplus \tilde{B} &= 
			\left(A_0\otimes I_{s}\otimes I_{T/ns}\right)+
			\left(B_0\otimes I_{p}\otimes I_{T/mp}\right)\\
			&=\left(A_0\otimes I_{R/n}\otimes I_{T/R}\right)+
			\left(B_0\otimes I_{R/m}\otimes I_{T/R}\right)\\
			&=\left(A_0\otimes I_{R/n}+B_0\otimes I_{R/m}\right)
			\otimes I_{T/R}.
		\end{split}
	\end{equation}

	Similarly, we have
	\begin{equation}\label{2.606}
		\begin{split}
			{A} \lplus {B}
			&=\left(A_0\otimes I_{R/n}+B_0\otimes I_{R/m}\right)
			\otimes I_{S/R}.
		\end{split}
	\end{equation}

	\eqref{2.605} and \eqref{2.606} imply $
	\tilde{A}\lplus \tilde{B}\sim A\lplus B$.

\end{proof}

Given $\mu\in\Q_{+}$, define the quotient space $\Sigma_{\mu}$ as
\begin{align}\label{2.7}
	\Sigma_{\mu}:={\cal M}_{\mu}/\sim=\{\A|A\in{\cal M}_{\mu}\}.
\end{align}

According to Theorem \ref{t2.8}, the operation $\lplus$ can be extended to $\Sigma_{\mu}$
as
\begin{align}\label{2.9}
\A\lplus \B:=\left<A\lplus B\right>,\quad \forall \A,\B\in \Sigma_{\mu}.
\end{align}

\begin{thm}\label{t2.9} 
	Using the definition in (\ref{2.9}), the quotient space
	$\left(\Sigma_{\mu}, \lplus\right)$ is a vector space.
\end{thm}

\begin{prp}\label{p4.5}
 Consider the algebraic system $\left({\cal M}, \ltimes \right)$.
The LIE $\sim$ is a congruence with respect to $\ltimes$.
\end{prp}

\begin{proof}
	The proof of this Proposition is similar to the one for Theorem \ref{t2.8}.

	Arbitrarily given two identity equivalence classes $\left<\bar A\right>,\left<\bar B\right>\subset{\cal M}$,
	by Corollary \ref{c2.400}, we have $\left<\bar A\right>=\{A_0\otimes I_s|s=0,1,2,\cdots\}$ and 
	$\left<\bar B\right>=\{B_0\otimes I_t|t=0,1,2,\cdots\}$, where $A_0$ and $B_0$ are the unique irreducible
	elements of $\left<\bar A\right>$ and $\left<\bar B\right>$, respectively.
	Arbitrarily chosen $A_0\otimes I_{s_1},A_0\otimes I_{s_2}\in\left<\bar A\right>$,  and
	$B_0\otimes I_{t_1},B_0\otimes I_{t_2}\in\left<\bar B\right>$, we need to prove $(A_0\otimes I_{s_1})
	\ltimes (B_0\otimes I_{t_1})\sim (A_0\otimes I_{s_2})\ltimes (B_0\otimes I_{t_2})$.
	Denote the size of $A_0$ and $B_0$ by $m\times n$ and $p\times q$, respectively.
	Set $\lcm(ns_1,pt_1)=:T_1$, $\lcm(ns_2,pt_2)=:T_2$, then
	\begin{equation}
		\begin{split}
			&\left((A_0\otimes I_{s_1})\ltimes (B_0\otimes I_{t_1})\right)\otimes I_{T_2}\\
			=&\left(\left(A_0\otimes I_{s_1}\otimes I_{\frac{T_1}{ns_1}}\right)\left(B_0\otimes I_{t_1}\otimes I_{\frac{T_1}{pt_1}}\right)\right)\otimes\left(I_{T_2}I_{T_2}\right)\\
			=&\left(\left(A_0\otimes I_{\frac{T_1}{n}}\right)\left(B_0\otimes I_{\frac{T_1}{p}}\right)\right)\otimes\left(I_{T_2}I_{T_2}\right)\\
			=&\left(A_0\otimes I_{\frac{T_1}{n}}\otimes I_{T_2}\right)\left(B_0\otimes I_{\frac{T_1}{p}}\otimes I_{T_2}\right)\\
			=&\left(A_0\otimes I_{\frac{T_2}{n}}\otimes I_{T_1}\right)\left(B_0\otimes I_{\frac{T_2}{p}}\otimes I_{T_1}\right)\\
			=&\left(\left(A_0\otimes I_{\frac{T_2}{n}}\right)\left(B_0\otimes I_{\frac{T_2}{p}}\right)\right)\otimes I_{T_1}\\
			=&\left(\left(A_0\otimes I_{s_2}\otimes I_{\frac{T_2}{ns_2}}\right)\left(B_0\otimes I_{t_2}\otimes I_{\frac{T_2}{pt_2}}\right)\right)\otimes I_{T_1}\\
			=&\left((A_0\otimes I_{s_2})\ltimes (B_0\otimes I_{t_2})\right)\otimes I_{T_1},
		\end{split}
	\end{equation}
	i.e., $(A_0\otimes I_{s_1})\ltimes (B_0\otimes I_{t_1})\sim (A_0\otimes I_{s_2})\ltimes (B_0\otimes I_{t_2})$,
	which completes the proof.
\end{proof}

According to Proposition \ref{p4.5} and Remark \ref{r2.6},
$\ltimes$ is well defined over the quotient product space $\Sigma_{\mu_1}
\times\Sigma_{\mu_2}$ for any $\mu_1,\mu_2\in\Q_{+}$, and maps $\Sigma_{\mu_1}\times\Sigma_{\mu_2}$ to
$\Sigma_{\mu_1\mu_2}$ according to the following operation:
For all $\A\in\Sigma_{\mu_1},\B\in{\Sigma}_{\mu_2}$ and all $\tilde A\in\A,\tilde B\in\B$,
$\A\ltimes\B=:\left<\tilde A\ltimes\tilde B\right>$.

\begin{thm}\label{t7.2} 
	The algebraic system $(\Sigma_1,\lplus,\ltimes)$ with Lie bracket $[\cdot,\cdot]$ defined in (\ref{7.4})
	\begin{align}\label{7.4}
[\A,\B]:=\A\ltimes \B\lminus \B\ltimes \A
\end{align}
	is a Lie algebra, which is called a multi-dimensional Lie algebra.
\end{thm}

\section{Basis for the quotient space}

In this section, we show that for each positive rational number $\mu$, the quotient space
$\left(\Sigma_{\mu}, \lplus\right)$ 
is of countably infinite dimension, give a generator for the space, and
then use the generator to construct a basis for the space.

\begin{prp}\label{p2.5}
	For every positive rational number $\mu=p/q$, where $p,q\in\Z_{+}$ are co-prime,
	the vector space $\left(\Sigma_{\mu},\lplus\right)$ is of countably infinite dimension,
	and has a generator
	\begin{equation}\label{2.11}
		\bigcup_{k\in\Z_{+}}\bigcup_{\begin{subarray}{c}1\le i_{kp}\le kp\\1\le j_{kq}\le kq\end{subarray}}
		\left\{\left<E_{i_{kp}j_{kq}}
		^{kp\times kq}\right>\right\},
	\end{equation}
	where and hereinafter $E^{kp\times kq}_{i_{kp}j_{kq}}$ denotes the $kp\times kq$ matrix with the $(i_{kp},j_{kq})$-th
	entry equal to $1$ and all other entries $0$.
\end{prp}

\begin{proof}
	Arbitrarily chosen $\A$ in $\Sigma_{\mu}$ such that $A=(a_{ij})_{1\le i\le k_0p,1\le j\le k_0q}
	\in{\cal M}_{k_0p\times k_0q}$ 
	is the unique irreducible matrix of $\A$, we have
	$$\A=\mathop{\lplus}\limits_{\substack{1\le i\le k_0p\\1\le j\le k_0q}}
		a_{ij}\left<E_{ij}^{k_0p\times k_0q}\right>$$
	by Theorem \ref{t2.8}. Hence the set \eqref{2.11} of countably infinite cardinality  is a generator of the vector space 
	$\left(\Sigma_{\mu},\lplus\right)$, and the dimension of $\left(\Sigma_{\mu},\lplus\right)$
	is at most countably infinite.

	We claim that the dimension of $\left(\Sigma_{\mu},\lplus\right)$ is countably infinite.
	Suppose the contrary: $\left(\Sigma_{\mu},\lplus\right)$ is of finite dimension, then
	there exists a positive integer $l$ such that 
	$$\bigcup_{1\le k\le l}\bigcup_{\begin{subarray}{c}1\le i_{kp}\le kp\\1\le j_{kq}\le kq\end{subarray}}
		\left\{\left<E_{i_{kp}j_{kq}}
	^{kp\times kq}\right>\right\}$$ is a generator of
	$\left(\Sigma_{\mu},\lplus\right)$. Hence for all matrices $A$ in ${\cal M}_{(l+1)p\times (l+1)q}$,
	$\left<A\right>$ can be generated by 
	\begin{equation}\label{2.12}
		\bigcup_{\substack{k\le l\\k|(l+1)}}
		\bigcup_{\substack{1\le i_{kp}\le kp\\1\le j_{kq}\le kq}}
		\left\{\left<E_{i_{kp}j_{kq}}^{kp\times kq}\right>\right\}.
	\end{equation}

	It is not difficult to obtain that $\left<E_{12}^{(l+1)p\times(l+1)q}\right>$ cannot be
	generated by \eqref{2.12}, which is a contradiction, and hence completes the proof.
\end{proof}

Note that for each positive integer $l$, the subset
$$\bigcup_{k=l}^{\infty}\bigcup_{\begin{subarray}{c}1\le i_{kp}\le kp\\1\le j_{kq}\le kq\end{subarray}}
\left\{\left<E_{i_{kp}j_{kq}}^{kp\times kq}\right>\right\}$$ of \eqref{2.11} is
a generator of the vector space $\left(\Sigma_{\mu},\lplus\right)$, since
for all positive integers $k_1,k_2$ such that $k_1$ divides $k_2$, 
$$\bigcup_{\begin{subarray}{c}1\le i_{k_1p}\le k_1p\\1\le j_{k_1q}\le k_1q\end{subarray}}
\left\{\left<E_{i_{k_1p}j_{k_1q}}^{k_1p\times k_1q}\right>\right\}$$ is
generated by $$\bigcup_{\begin{subarray}{c}1\le i_{k_2p}\le k_2p\\1\le j_{k_2q}\le k_2q\end{subarray}}
\left\{\left<E_{i_{k_2p}j_{k_2q}}^{k_2p\times k_2q}\right>\right\}.$$
That is, \eqref{2.11} is not a basis of $\left(\Sigma_{\mu},\lplus\right)$.
Next we use \eqref{2.11} to construct a basis for 
the vector space $\left(\Sigma_{\mu},\lplus\right)$.

\begin{thm}\label{t2.10}
	For every positive rational number $\mu=p/q$, where $p,q\in\Z_{+}$ are co-prime,
	the following countably infinite set of equivalence classes 
	\begin{equation}\label{2.13}
		{\sf D}_{\mu}\cup{\sf N}_{\mu}
	\end{equation}
	is a basis for the vector space  $\left(\Sigma_{\mu},\lplus\right)$,
	where\footnote{
	Note that in \eqref{2.14}, $E^{p\times q}_{kl}\otimes E^{i\times i}_{jj}$
	is the unique irreducible element of $\left<E^{p\times q}_{kl}\otimes E^{i\times i}_{jj}\right>$, and
	$E^{p\times q}_{kl}\otimes E^{i\times i}_{j_1j_2}$ is the unique irreducible element 
	of $\left<E^{p\times q}_{kl}\otimes E^{i\times i}_{j_1j_2}\right>$.}
	\begin{equation}\label{2.14}
		\begin{split}
			&{\sf D}_{\mu}=\bigcup_{i=1}^{\infty}\bigcup_{\substack{1\le j\le i\\\gcd(i,j)=1}}
			\bigcup_{\substack{1\le k\le p\\1\le l\le q}}
			\left\{\left<E^{p\times q}_{kl}\otimes E^{i\times i}_{jj}\right>\right\},\\
			&{\sf N}_{\mu}=\bigcup_{i=2}^{\infty}\bigcup_{\substack{1\le j_1,j_2\le i\\j_1\ne j_2\\\gcd(i,j_1,j_2)=1}}
			\bigcup_{\substack{1\le k\le p\\1\le l\le q}}
			\left\{\left<E^{p\times q}_{kl}\otimes E^{i\times i}_{j_1j_2}\right>\right\},\\
		\end{split}
	\end{equation}
	where $\gcd()$ denotes the greatest common divisor of numbers in $()$.
\end{thm}

\begin{proof}
	To prove this theorem, we need to show that \romannumeral1) \eqref{2.13} is a generator of
	$\left(\Sigma_{\mu},\lplus\right)$ and \romannumeral2)
	every finite number of elements of \eqref{2.13} are linearly independent.
	Fix co-prime positive integers $p$ and $q$.

	\romannumeral1): We next prove that for all positive integers $i,k,l,j_1,j_2$ satisfying 
	$1\le k\le p$, $1\le l\le q$ and $1\le j_1,j_2\le i$,
	$\left<E^{p\times q}_{kl}\otimes E^{i\times i}_{j_1j_2}\right>$ is generated by \eqref{2.13}.
		We divide the proof into two cases that $j_1=j_2=:j$ and $j_1\ne j_2$.

	$j_1=j_2$:
	
	If $\gcd(i,j)=1$, then $\left<E^{p\times q}_{kl}\otimes E^{i\times i}_{jj}\right>\in
	{\sf D}_{\mu}$, and $\left<E^{p\times q}_{kl}\otimes E^{i\times i}_{jj}\right>$ is generated by 
	\eqref{2.13}. Next we assume that $\gcd(i,j)>1$. Write 
	$$
	\begin{array}[]{l}
		f_1+f_2+\cdots+f_{m'}=j,\\
		g_1+g_2+\cdots+g_m=j-1,\\
	\end{array}
	$$
	where
	\begin{equation}
		\begin{split}
			&f_1 = \gcd(i,j),\\
			&f_2 = \gcd(i,j-f_1),\\
			&\vdots\\
			&f_{m'} = \gcd\left(i,j-\sum\nolimits_{\a=1}^{m'}f_{\a}\right),\\
			&g_1 = \gcd(i,j-1),\\
			&g_2 = \gcd(i,j-1-g_1),\\
			&\vdots\\
			&g_{m} = \gcd\left(i,j-1-\sum\nolimits_{\a=1}^{m-1}g_{\a}\right).\\
		\end{split}
	\end{equation}
	We have
	\begin{equation}
		\begin{split}
			E^{i\times i}_{jj} =
			\left(\mathop{\lplus}_{n'=1}^{m'}E^{\bar i'_{n'}\times \bar i'_{n'}}_{\bar j'_{n'}\bar j'_{n'}}\right)
			\lminus
			\left(\mathop{\lplus}_{n=1}^{m}E^{\bar i_n\times \bar i_n}_{\bar j_n\bar j_n}\right),
		\end{split}
	\end{equation}
	where $\bar i'_{n'}=i/f_{n'}$, $\bar j'_{n'}=(j-\sum_{\a=1}^{n'-1}f_{\a}
	)/f_{n'}$, $n'=1,2,\cdots,m'$;
	$\bar i_{n}=i/g_{n}$, $\bar j_{n}=(j-1-\sum_{\a=1}^{n-1}g_{\a})/g_{n}$, $n=1,2,\cdots,m$.

	Then $\bar j'_{n'}\le\bar i'_{n'}$, $\gcd(\bar i'_{n'},\bar j'_{n'})=1$,
	$\bar j_n\le\bar i_n$, $\gcd(\bar i_n,\bar j_n)=1$,
	$\left<E^{p\times q}_{kl}\otimes E^{\bar i'_{n'}\times \bar i'_{n'}}_{\bar j'_{n'}\bar j'_{n'}}\right>\in{\sf D}_{\mu}$,
	$\left<E^{p\times q}_{kl}\otimes E^{\bar i_n\times \bar i_n}_{\bar j_n\bar j_n}\right>\in{\sf D}_{\mu}$,
	$k=1,2,\cdots,p$, $l=1,2,\cdots,q$, $n'=1,2,\cdots,m'$, $n=1,2,\cdots,m$, and by Theorem \ref{t2.8},
	\begin{equation}\label{2.21}
		\begin{split}
			\left<E^{p\times q}_{kl}\otimes E^{i\times i}_{jj}\right>
			=
			&\left(\mathop{\lplus}_{n'=1}^{m'}\left<E^{p\times q}_{kl}\otimes E^{\bar i'_{n'}\times \bar i'_{n'}}_{\bar j'_{n'}\bar j'_{n'}}\right>\right)
			\lminus\\
			&\left(\mathop{\lplus}_{n=1}^{m}\left<E^{p\times q}_{kl}\otimes E^{\bar i_n\times \bar i_n}
			_{\bar j_n\bar j_n}\right>\right)\\
		\end{split}
	\end{equation}
	is generated by ${\sf D}_{\mu}$.

	$j_1\ne j_2$:

	Similar to the former case, if $\gcd(i,j_1,j_2)=1$, then
	$\left<E^{p\times q}_{kl}\otimes E^{i\times i}_{j_1j_2}\right>\in
	{\sf N}_{\mu}$, and $\left<E^{p\times q}_{kl}\otimes E^{i\times i}_{j_1j_2}\right>$ is generated by 
	\eqref{2.13}. Next we consider $\gcd(i,j_1,j_2)>1$, and assume that $j_1<j_2$ without loss of generality. Write 
	$$
	\begin{array}[]{l}
		f_1+f_2+\cdots+f_{m'}=j_1,\\
		g_1+g_2+\cdots+g_m=j_1-1,\\
	\end{array}
	$$
	where
	\begin{equation}
		\begin{split}
			&f_1 = \gcd(i,j_1,j_2),\\
			&f_2 = \gcd(i,j_1-f_1,j_2-f_1),\\
			&\vdots\\
			&f_{m'} = \gcd\left(i,j_1-\sum\nolimits_{\a=1}^{m'}f_{\a},j_2-\sum\nolimits_{\a=1}^{m'}f_{\a}\right),\\
			&g_1 = \gcd(i,j_1-1,j_2-1),\\
			&g_2 = \gcd(i,j_1-1-g_1,j_2-1-g_1),\\
			&\vdots\\
			&g_{m} = \gcd\left(i,j_1-1-\sum\nolimits_{\a=1}^{m-1}g_{\a},j_2-1-\sum\nolimits_{\a=1}^{m-1}g_{\a}\right).\\
		\end{split}
	\end{equation}


	We have
	\begin{equation}
		\begin{split}
			E^{i\times i}_{j_1j_2} =
			\left(\mathop{\lplus}_{n'=1}^{m'}E^{\bar i'_{n'}\times \bar i'_{n'}}_{\bar j'_{1n'}\bar j'_{2n'}}\right)
			\lminus
			\left(\mathop{\lplus}_{n=1}^{m}E^{\bar i_n\times \bar i_n}_{\bar j_{1n}\bar j_{1n}}\right),
		\end{split}
	\end{equation}
	where $\bar i'_{n'}=i/f_{n'}$, $\bar j'_{1n'}=(j_1-\sum_{\a=1}^{n'-1}f_{\a})/f_{n'}$,
	$\bar j'_{2n'}=(j_2-\sum_{\a=1}^{n'-1}f_{\a})/f_{n'}$,
	$n'=1,2,\cdots,m'$;
	$\bar i_{n}=i/g_{n}$, $\bar j_{1n}=(j_1-1-\sum_{\a=1}^{n-1}g_{\a})/g_{n}$,
	$\bar j_{2n}=(j_2-1-\sum_{\a=1}^{n-1}g_{\a})/g_{n}$, $n=1,2,\cdots,m$.

	Then $\bar j'_{1n'},\bar j'_{2n'}\le\bar i'_{n'}$, $\gcd(\bar i'_{n'},\bar j'_{1n'},\bar j'_{2n'})=1$,
	$\bar j_{1n},\bar j_{2n}\le\bar i_n$, $\gcd(\bar i_n,\bar j_{1n},\bar j_{2n})=1$,
	$\left<E^{p\times q}_{kl}\otimes E^{\bar i'_{n'}\times \bar i'_{n'}}_{\bar j'_{1n'}\bar j'_{2n'}}\right>\in{\sf N}_{\mu}$,
	$\left<E^{p\times q}_{kl}\otimes E^{\bar i_n\times \bar i_n}_{\bar j_{1n}\bar j_{2n}}\right>\in{\sf N}_{\mu}$,
	$k=1,2,\cdots,p$, $l=1,2,\cdots,q$, $n'=1,2,\cdots,m'$, $n=1,2,\cdots,m$, and by Theorem \ref{t2.8},
	\begin{equation}\label{2.22}
		\begin{split}
			\left<E^{p\times q}_{kl}\otimes E^{i\times i}_{j_1j_2}\right>
			=
			&\left(\mathop{\lplus}_{n'=1}^{m'}\left<E^{p\times q}_{kl}\otimes E^{\bar i'_{n'}\times \bar i'_{n'}}_{\bar j'_{1n'}\bar j'_{2n'}}\right>\right)
			\lminus\\
			&\left(\mathop{\lplus}_{n=1}^{m}\left<E^{p\times q}_{kl}\otimes E^{\bar i_n\times \bar i_n}
			_{\bar j_{1n}\bar j_{2n}}\right>\right)\\
		\end{split}
	\end{equation}
	is generated by ${\sf N}_{\mu}$.

	Based on the above, we have \eqref{2.11} is generated by \eqref{2.13}, then by Proposition 
	\ref{p2.5}, \eqref{2.13} is a generator of $\left(\Sigma_{\mu},\lplus\right)$.
	Besides, by \eqref{2.21} and \eqref{2.22} we have every element of $\left(\Sigma_{\mu},\lplus\right)$
	is represented uniquely as a linear combination of finitely many elements of \eqref{2.13}.

	\romannumeral2): To prove that every finite number of elements of \eqref{2.13} are linearly independent,
	we need to prove that every finite number of elements of ${\sf D}_{\mu}$ are linearly independent
	and every finite number of elements of ${\sf N}_{\mu}$ are linearly independent, because 
	in each $E^{i\times i}_{jj}$, only diagonal entries can be nonzero, and in each $E^{i\times i}_{j_1j_2}$
	with $j_1\ne j_2$, only nondiagonal entries can be nonzero.

	To prove that every finite number of elements of ${\sf D}_{\mu}$ are linearly independent, we need to prove
	that for all positive integers $k,l,n$ satisfying that $1\le k\le p$ and $1\le l\le q$,
	\begin{equation}\label{2.15}
		\bigcup_{i=1}^{n}\bigcup_{\substack{1\le j\le i\\\gcd(i,j)=1}}
		\left\{\left<E^{p\times q}_{kl}\otimes E^{i\times i}_{jj}\right>\right\}
	\end{equation}
	are linearly independent. Write $\lcm(1,2,\cdots,n)=:t$, then \eqref{2.15} equals
	\begin{equation}\label{2.16}
		\bigcup_{i=1}^{n}\bigcup_{\substack{1\le j\le i\\\gcd(i,j)=1}}
		\left\{\left<E^{p\times q}_{kl}\otimes E^{i\times i}_{jj}\otimes I_{t/i}\right>\right\}.
	\end{equation}
	
	What is left is to prove the set
	\begin{equation}\label{2.17}
		\left\{jt/i|1\le i\le n,1\le j\le i,\gcd(i,j)=1\right\} 
	\end{equation}
	has the same cardinality as \eqref{2.16}, because for each 
	$E^{i\times i}_{jj}\otimes I_{t/i}$, the $(jt/i,jt/i)$-th entry equals $1$,
	and either $jt/i=t$ or the $(s,s)$-th entry equals $0$ for each ineger $s>jt/i$.

	Suppose there exist positive integers $i_1,j_1,i_2,j_2$ such that 
	$1\le j_1\le i_1$, $1\le j_2\le i_2$, $\gcd(i_1,j_1)=\gcd(i_2,j_2)=1$, 
	and $j_1t/i_1=j_2t/i_2$, then $i_1=i_2$ and $j_1=j_2$. Hence the set \eqref{2.17}
	has the same cardinality as the set \eqref{2.16}.

	To prove that every finite number of elements of ${\sf N}_{\mu}$ are linearly independent,
	similar to the case for ${\sf D}_{\mu}$, we need to prove
	that for all positive integers $k,l,n$ satisfying that $n\ge 2$, $1\le k\le p$ and $1\le l\le q$,
	\begin{equation}\label{2.18}
		\bigcup_{i=2}^{n}\bigcup_{\substack{1\le j_1,j_2\le i\\j_1\ne j_2\\\gcd(i,j_1,j_2)=1}}
		\left\{\left<E^{p\times q}_{kl}\otimes E^{i\times i}_{j_1j_2}\right>\right\}
	\end{equation}
	are linearly independent. Write $\lcm(2,\cdots,n)=:t$, then \eqref{2.18} equals
	\begin{equation}\label{2.19}
		\bigcup_{i=2}^{n}\bigcup_{\substack{1\le j_1,j_2\le i\\j_1\ne j_2\\\gcd(i,j_1,j_2)=1}}
		\left\{\left<E^{p\times q}_{kl}\otimes E^{i\times i}_{j_1j_2}\otimes I_{t/i}\right>\right\}.
	\end{equation}
	
	What is left is to prove the set
	\begin{equation}\label{2.20}
		\begin{split}
			\{(j_1t/i,j_2t/i)| &1\le i\le n,1\le j_1,j_2\le i,j_1\ne j_2,\\
			&\gcd(i,j_1,j_2)=1\} 
		\end{split}
	\end{equation}
	has the same cardinality as \eqref{2.19}, because for each
	$E^{i\times i}_{j_1j_2}\otimes I_{t/i}$, the
	$(j_1t/i,j_2t/i)$-th entry equals $1$, and either $j_1t/i=t$, or $j_2t/i=t$,
	or the $(j_1t/i+s,j_2t/i+s)$-th entry equals $0$ for each positive integer $s$.

	Suppose there exist positive integers $i_1,j_1^1,j_2^1,i_2,j_1^2,j_2^2$ such that 
	$1\le j_1^1,j_2^1\le i_1$, $1\le j_1^2,j_2^2\le i_2$, $j_1^1\ne j_2^1$, $j_1^2\ne j_2^2$,
	$\gcd(i_1,j_1^1,j_2^1)=\gcd(i_2,j_1^2,j_2^2)=1$, 
	$j_1^1t/i_1=j_1^2t/i_2$, and $j_2^1t/i_1=j_2^2t/i_2$, then $i_1=i_2$, $j_1^1=j_1^2$ and $j_2^1=j_2^2$.
	This follows from the discussion below.
	\romannumeral1) Suppose $\gcd(i_1,j_1^1)=1$ and $i_1\ne i_2$, then $i_1|i_2$, $i_2/i_1=:t>1$, and 
	$\gcd(i_2,j_1^2,j_2^2)=\gcd(ti_1,tj_1^1,tj_2^1)=t>1$, which is a contradiction. 
	\romannumeral2) Suppose $\gcd(i_1,j_1^1)>1$, $i_1\ne i_2$ and $\gcd(i_2,j_1^2)=1$, then similarly 
	$i_2|i_1$, and $\gcd(i_1,j_1^1,j_2^1)=i_1/i_2>1$, which is also a contradiction.
	\romannumeral3) Suppose $\gcd(i_1,j_1^1)=:t>1$, $i_1\ne i_2$ and $\gcd(i_2,j_1^2)=:s>1$,
	then $\gcd(i_1/t,j_1^1/t)=\gcd(i_2/s,j_1^2/s)=1$, $i_1/t=i_2/s$, $\gcd(t,s)|j_2^1$, $\gcd(t,s)|j_2^2$, 
	$t\ne s$ (if $t=s$, then $i_1=i_2$, a contradiction),
	$\gcd(i_1,j_1^1,j_2^1)\ge t/\gcd(t,s)>1$ or $\gcd(i_2,j_1^2,j_2^2)\ge s/\gcd(t,s)>1$,
	which is again a contradiction.
	Hence the set \eqref{2.20} has the same cardinality as the set \eqref{2.19}, which completes the proof.
\end{proof}

\section{Inner product for the quotient space}

We have shown a basis for the quotient space $(\Sigma_{\mu},\lplus)$, where $\mu\in\Q_{+}$.
Next we endow the quotient space with an inner product which is a generalization of the conventional inner product
of matrices. For ${\cal M}_{m\times n}$, the conventional inner product is defined as
for all matrices $A=(a_{ij})_{m\times n},B=(b_{ij})_{m\times n}$ both in
${\cal M}_{m\times n}$,
$\left<A,B\right>:=\sum_{1\le i\le m,1\le j\le n}a_{ij}b_{ij}$.

The inner product for $(\Sigma_{\mu},\lplus)$ is defined as for all $\A,\B\in\Sigma_{\mu}$,
where $A\in{\cal M}_{\mu p\times p},B\in{\cal M}_{\mu q\times q}$ are the unique irreducible elements of $\A,\B$,
respectively, 
\begin{equation}
	\left<\A,\B\right>:=\left<A\otimes I_{t/(\mu p)},B\otimes I_{t/(\mu q)}\right>,
	\label{eqn_innerproduct}
\end{equation}
where $t=\lcm(\mu p,\mu q)$.

The norm induced by the inner product \eqref{eqn_innerproduct} is defined as
for all $\A\in\Sigma_{\mu}$,
\begin{equation}
	\|\A\|:=\sqrt{\left<\A,\A\right>}.
	\label{eqn_norm}
\end{equation}

The metric induced by the norm \eqref{eqn_norm} is defined as
for all $\A,\B\in\Sigma_{\mu}$,
\begin{equation}
	\text{d}(\A,\B):=\|\A\lminus\B\|.
	\label{eqn_metric}
\end{equation}

Next for the inner product space $(\Sigma_{\mu},\left<\cdot,\cdot\right>)$, where $\mu\in\Q_{+}$,
we give a Cauchy sequence
that does not converge, showing that this inner product space is not a Hilbert space.
 
Defined a set of mappings $\d_{n}:{\cal M}\to {\cal M}$ as for all $A=(a_{ij})_{p\times q}\in{\cal M}_{p\times q}$,
$\d_n(A)=(b_{ij})_{p\times q}$, where 
\begin{align}
	b_{ij} = 
	\left\{
	\begin{array}[]{ll}
		a_{ij}, &\text{ if }a_{ij}\ne 0,\\
		\frac{1}{2^{\frac{2^{n-1}}{\ln 2}}}, &\text{ otherwise.}
	\end{array}
	\right.
	\label{}
\end{align}

Consider the following sequence
\begin{equation}
	\left<A_1\right>,\left<A_2\right>,\dots\in\Sigma_{\mu},
	\label{eqn:Cauchyseq.}
\end{equation}
where $\mu\in\Q_{+}$, all entries of  $A_1\in{\cal M}_{\mu}$ are nonzero, 
for all positive integers $n>1$, $A_n=\d_n(A_{n-1}\otimes I_2)$.
Then for each $n\in\Z_{+}$, $A_n$ is the unique irreducible element of $\left<A_n\right>$.

For example, if $\mu=1/2$ and $$A_1=\begin{bmatrix}1& 2\end{bmatrix},$$ then
	\begin{align*}
&A_2=
\begin{bmatrix}
	1 & \frac{1}{2^{\frac{2}{\ln 2}}} & 2 & \frac{1}{2^{\frac{2}{\ln 2}}}\\
	\frac{1}{2^{\frac{2}{\ln 2}}} & 1 & \frac{1}{2^{\frac{2}{\ln 2}}} & 2
\end{bmatrix},\\
&A_3=
\begin{bmatrix}
	1 & \frac{1}{2^{\frac{2^2}{\ln 2}}} & \frac{1}{2^{\frac{2}{\ln 2}}} & \frac{1}{2^{\frac{2^2}{\ln 2}}} & 2 & \frac{1}{2^{\frac{2^2}{\ln 2}}} & \frac{1}{2^{\frac{2}{\ln 2}}} & \frac{1}{2^{\frac{2^2}{\ln 2}}} \\
	\frac{1}{2^{\frac{2^2}{\ln 2}}} & 1 & \frac{1}{2^{\frac{2^2}{\ln 2}}} & \frac{1}{2^{\frac{2}{\ln 2}}} & \frac{1}{2^{\frac{2^2}{\ln 2}}} & 2 & \frac{1}{2^{\frac{2^2}{\ln 2}}} & \frac{1}{2^{\frac{2}{\ln 2}}} \\
	\frac{1}{2^{\frac{2}{\ln 2}}} & \frac{1}{2^{\frac{2^2}{\ln 2}}} & 1 & \frac{1}{2^{\frac{2^2}{\ln 2}}} & \frac{1}{2^{\frac{2}{\ln 2}}} & \frac{1}{2^{\frac{2^2}{\ln 2}}} & 2 & \frac{1}{2^{\frac{2^2}{\ln 2}}}\\
	\frac{1}{2^{\frac{2^2}{\ln 2}}} & \frac{1}{2^{\frac{2}{\ln 2}}} & \frac{1}{2^{\frac{2^2}{\ln 2}}} & 1 & \frac{1}{2^{\frac{2^2}{\ln 2}}} & \frac{1}{2^{\frac{2}{\ln 2}}} & \frac{1}{2^{\frac{2^2}{\ln 2}}} & 2\\
\end{bmatrix},\\
&\cdots.
\end{align*}

\begin{prp}\label{thm_Cauchyseq.}
	The sequence $\left<A_1\right>,\left<A_2\right>,\dots\in\Sigma_{\mu}$ defined in \eqref{eqn:Cauchyseq.}
	is a Cauchy sequence that does not converge.
\end{prp}

\begin{proof}
	Let $A_1$ be of size $p\times q$. Then for each $n\in\Z_{+}$,
	$A_n$ is of size $(2^{n-1}p)\times (2^{n-1}q)$.

	We first prove that the sequence \eqref{eqn:Cauchyseq.} is a Cauchy sequence. 
	For any positive integer $n$, 
	we have
	\begin{align*}
		\text{d}(\left<A_n\right>,\left<A_{n+1}\right>)
		&= \sqrt{2^{2n-1}pq\left( \frac{1}{2^{\frac{2^n}{\ln 2}}} \right)^2}\\
		&= \sqrt{pq 2^{2n-1-\frac{2^{n+1}}{\ln 2}}}\\
		&\le \sqrt{pq 2^{-1-\frac{2}{\ln 2}-n^2\ln2}},
	\end{align*}
	since $2^n\ge 1 + n\ln 2 + \frac{n^2}{2}(\ln 2)^2$ (the first three terms of its Taylor series).
	Denote $0<\sqrt{2^{-\ln 2}}=:a<1$. Then for any positive integers $n,m$, 
	we have 
	\begin{align*}
		&\text{d}(\left<A_n\right>,\left<A_{n+m}\right>) \\
		\le& \text{d}(\left<A_n\right>,\left<A_{n+1}\right>) +
		\text{d}(\left<A_{n+1}\right>,\left<A_{n+2}\right>) + \cdots \\&+
		\text{d}(\left<A_{n+m-1}\right>,\left<A_{n+m}\right>)\\
		= & \sqrt{ pq 2^{-1-\frac{2}{\ln 2}}}\left( a^{n^2} + a^{(n+1)^{2}} + \cdots + a^{(n+m-1)^2} \right)\\
		< & \sqrt{pq 2^{-1-\frac{2}{\ln 2}}} \left( a^{n^2} + a ^{n^2+1} + \cdots + a^{(n+m-1)^2} \right)\\
		= & \sqrt{pq 2^{-1-\frac{2}{\ln 2}}} \frac{a^{n^2}\left( 1-a^{(n+m-1)^2-n^2+1} \right)}{1-a}\\
		\le & \sqrt{pq 2^{-1-\frac{2}{\ln 2}}} \frac{a^{n^2}}{1-a}.
	\end{align*}
	Hence for any sufficiently small positive real number $\e$,
	choose positive integer $N>\sqrt{\log_a\left( \frac{\e(1-a)}{\sqrt{pq2^{-1-\frac{2}{\ln 2}}}} \right)}>0$,
	we have for all positive integers $n,m>N$, $\text{d}(\left<A_n\right>,\left<A_m\right>)<\e$, i.e.,
	the sequence \eqref{eqn:Cauchyseq.} is a Cauchy sequence.

	Second we prove that the sequence \eqref{eqn:Cauchyseq.} does not converge in $\Sigma_{\mu}$.
	For an arbitrarily given $\left<A_0\right>\in\Sigma_{\mu}$, where $A_0$ is the unique irreducible element of
	$\left<A_0\right>$, we claim that the sequence \eqref{eqn:Cauchyseq.} does not converge to $\left<A_0\right>$.
	We divide the proof of the claim into three cases.
	\romannumeral1) If $\left<A_0\right>$ equals $\left<A_m\right>$ for some $m\in\Z_{+}$,
	then $A_0=A_m$, and $\text{d}(\left<A_0\right>,\left<A_n\right>)>\text{d}(\left<A_0\right>,\left<A_{n-1}\right>)>
	\cdots>\text{d}(\left<A_0\right>,\left<A_{m+1}\right>)>
	\frac{1}{2^{\frac{2^m}{\ln 2}}}$
	for any positive integer $n > m+1$,
	i.e., in this case the sequence \eqref{eqn:Cauchyseq.} does not converge to $\left<A_0\right>$.
	\romannumeral2) If $\left<A_0\right>$ does not equal $\left<A_m\right>$ for any $m\in\Z_{+}$
	and $A_0=(a^0_{ij})_{(2^{n-1}p)\times (2^{n-1}q)}$ is of the same size as 
	$A_n=(a^n_{ij})_{(2^{n-1}p)\times (2^{n-1}q)}$ for some $n\in\Z_{+}$, then
	there exist integers $1\le i'\le 2^{n-1}p$ and $1\le j'\le 2^{n-1}q$ such that
	$a^0_{i'j'}\ne a^n_{i'j'}$, and
	$\text{d}(\left<A_0\right>,\left<A_r\right>)>\text{d}(\left<A_0\right>,\left<A_n\right>)\ge
	|a^0_{i'j'}- a^n_{i'j'}|>0$ for any positive integer $r>n$, i.e., 
	the sequence \eqref{eqn:Cauchyseq.} does not converge to $\left<A_0\right>$ either. \romannumeral3) If
	$\left<A_0\right>$ does not equal $\left<A_n\right>$ for any $n\in\Z_{+}$
	and $A_0$ is not of the same size as $A_n$ for any $n\in\Z_{+}$, let $A_0$ be of size $s\times t$,
	then $A_0\otimes I_{2^{n-1}p}$ and $A_n\otimes I_{s}$ are of the same size $(2^{n-1}sp)\times(2^{n-1}tp)$
	and $A_0\otimes I_{2^{n-1}p}$ does not equal $A_n\otimes I_{s}$ for any $n\in\Z_{+}$.
	It is because if $A_0\otimes I_{2^{m-1}p}=A_m\otimes I_{s}$ for some $m\in\Z_{+}$, then
	$\left<A_0\right>=\left<A_m\right>$, which is a contradiction.
	For each $n\in\Z_{+}$, the number of nonzero entries of $A_n\otimes I_s$ is $2^{2n-2}pqs$,
	and the number of nonzero entries of $A_0\otimes I_{2^{n-1}p}$ is at most $2^{n-1}stp$.
	Hence when $n$ is sufficiently large, $A_n\otimes I_s$ has much more nonzero entries than
	$A_0\otimes I_{2^{n-1}p}$ has. Chosen a sufficiently large $n$ satisfying that for 
	$A_n\otimes I_s=(b^n_{ij})_{(2^{n-1}sp)\times(2^{n-1}tp)}$ and
	$A_0\otimes I_{2^{n-1}p}=(b^0_{ij})_{(2^{n-1}sp)\times(2^{n-1}tp)}$, there exist
	$1\le i'\le 2^{n-1}sp$ and $1\le j'\le 2^{n-1}tp$ such that $b^n_{i'j'}\ne 0$ and $b^0_{i'j'}=0$,
	we then have $\text{d}(\left<A_0\right>,\left<A_n\right>)\ge|b^n_{i'j'}-b^0_{i'j'}|>0$.
	We furthermore have $\text{d}(\left<A_0\right>,\left<A_m\right>)\ge|b^n_{i'j'}-b^0_{i'j'}|>0$
	for any integer $m>n$, i.e., the sequence \eqref{eqn:Cauchyseq.} does not converge to $\left<A_0\right>$ either.



	
\end{proof}


\begin{thebibliography}{00}

\bibitem{che16STA}
	D. Cheng.
	Structure of Matrices under Equivalence.
	\url{http://arxiv.org/abs/1605.09523}



%
\bibitem{che01} D. Cheng, Semi-tensor product of matrices and its application to Morgan's problem, {\it Science in China, Series F: Information Science}, Vol. 44, No. 3, 195-212, 2001.
%

\end{thebibliography}
\end{document}